\documentclass[10pt]{article}

\title{\vskip-1.0em A gap theorem for the ZL-amenability constant of a finite group}
\def\shorttitle{ZL-amenability constant of a finite group}
\author{Yemon Choi}
\date{15th May 2015}

\RequirePackage{amsmath,amssymb,amsfonts}
\numberwithin{equation}{section}

\renewcommand{\emph}[1]{{\sl #1}\/} 
\newcommand{\st}{\mathbin{\colon}}
\newcommand{\defeq}{:=}
\newcommand{\dt}[1]{{\it#1}\/}

\newenvironment{YCnum}{%
\begin{enumerate}

}{\end{enumerate}\ignorespacesafterend}

\newcommand{\tp}{\mathbin{\otimes}}
\newcommand{\abs}[1]{\vert{#1}\vert}
\newcommand{\norm}[1]{\Vert{#1}\Vert}
\newcommand{\veps}{\varepsilon}
\newcommand{\iso}{\cong}

\newcommand{\Cplx}{\mathbb C}
\newcommand{\Nat}{\mathbb N}

\newcommand{\DZ}{\Delta_{\rm Z}}
\newcommand{\AMZL}{\operatorname{\rm AM}_{\rm Z}}
\newcommand{\ASS}{\operatorname{ass}}
\newcommand{\Irr}{\operatorname{Irr}}
\newcommand{\Conj}{\operatorname{Conj}}

\newcommand{\Aff}{\operatorname{Aff}}

\newcommand{\abar}{\overline{a}}

\newcommand{\gen}[1]{\langle#1\rangle}

\newcommand{\cL}{{\mathcal L}}

\newcommand{\cR}{{\mathcal R}}

\newcommand{\bF}{{\mathbf F}}

\newcommand{\bbI}{{\mathbb I}}
\newcommand{\bbJ}{{\mathbb J}}

\newcommand{\cZL}{\mathcal{ZL}}
\newcommand{\ZCG}{{\rm Z}\Cplx G}
\newcommand{\ZCGG}{{\rm Z}\Cplx (G\times G)}
\newcommand{\ZCH}{{\rm Z}\Cplx H}

\RequirePackage{amsthm}   

\newcounter{pulse}[section]
\numberwithin{pulse}{section}  

\newcommand{\thf}{\sc} 

\theoremstyle{plain}
\newtheorem{thm}[pulse]{\thf Theorem}

\newtheorem{prop}[pulse]{\thf Proposition}
\newtheorem{lem}[pulse]{\thf Lemma}
\newtheorem{cor}[pulse]{\thf Corollary}
\theoremstyle{definition}
\newtheorem{dfn}[pulse]{\thf Definition}

\newtheorem{eg}[pulse]{\thf Example}
\theoremstyle{remark}
\newtheorem{rem}[pulse]{\thf Remark}

\newcommand{\para}[1]{\paragraph{\thf#1}}

\newcounter{quasithm}
\renewcommand{\thequasithm}{\Alph{quasithm}}
\newenvironment{thmlike}[2][\unskip]{\paragraph{{\thf #2 \thequasithm} {\rm#1}.}\refstepcounter{quasithm}\it}{\rm\medskip}

\newcounter{question}
\newenvironment{qn}{\paragraph{\sc Question \thequestion.}\refstepcounter{question}}{\medskip}

\usepackage[usenames]{color}
\renewcommand{\dt}[1]{\textcolor{Bittersweet}{\textsf{#1}}}

\usepackage{sectsty}
\allsectionsfont{\sffamily}
\usepackage{fancyhdr}
\renewcommand{\para}[1]{\paragraph{#1.}}

\usepackage[noBBpl]{mathpazo}

\begin{document}

\maketitle

\begin{abstract}
It was shown in \cite{AzSaSp} that the ZL-amenability constant of a finite group is always at least~$1$, with equality if and only if the group is abelian. It was also shown in \cite{AzSaSp} that for any finite non-abelian group this invariant is at least $301/300$, but the proof relies crucially on a deep result of D. A. Rider on norms of central idempotents in group algebras.

Here we show that if $G$ is finite and non-abelian then its ZL-amenability constant is at least $7/4$, which is known to be best possible. We avoid use of Rider's result, by analyzing the cases where $G$ is just non-abelian, using calculations from \cite{ACS_AMZL-2cd}, and establishing a new estimate for groups with trivial centre.

(primary);
43A20,
43A62
(secondary).

\end{abstract}

\tableofcontents

\vfill\eject

\begin{section}{Introduction}
Given a finite group $G$, consider its complex group algebra $\Cplx G$, and recall that the centre $\ZCG$ is isomorphic to $\bigoplus_p \Cplx p$ where the sum is over all minimal idempotents in $\ZCG$. Now consider
\begin{equation}\label{eq:diagonal of ZC(G)}
\DZ(G) = \sum_p p \otimes p \in \ZCG\tp \ZCG\,.
\end{equation}
$\ZCG\tp\ZCG$ sits naturally inside $\Cplx G\tp\Cplx G \equiv \Cplx(G\times  G)$, and there is a natural $\ell^1$-norm $\norm{\cdot}_1$ that we can put on $\Cplx(G\times G)$. We define the \dt{ZL-amenability constant} of $G$, denoted by $\AMZL(G)$, to be $\norm{\DZ(G)}_1$.
The original reasons for studying $\AMZL(G)$ arose in connection with Banach algebras and (non-abelian) Fourier analysis; however, the problem which we address in this paper can be stated purely in terms of finite groups and the behaviour of their irreducible characters.

$\AMZL(G)$ seems to have first been studied explicitly in work of A.~Azimifard, E. Samei and N. Spronk, although they used different terminology and notation. One of their observations, paraphrased into our notation, is that $\AMZL(G)\geq 1$ with equality if and only if $G$ is abelian; this is proved using the Schur orthogonality relations and a clever use of an ``associated minorant''~$\ASS(G)$. (See Example~\ref{eg:AMZL(abelian)} and Proposition~\ref{p:ASS(non-abelian)} for further details.) The same paper also contains the following  ``gap'' result, which lies significantly deeper.

\begin{thmlike}[(Azimifard--Samei--Spronk, {\cite{AzSaSp}})]{Theorem}
\label{t:ASSgap}
There exists $\delta>0$ such that $\AMZL(G)\geq 1+\delta$ for every finite non-abelian group~$G$.
\end{thmlike}

The purpose of the present paper is to determine \emph{exactly} how big $\delta$ can be.

\begin{rem}
To provide background context, let us briefly mention why the authors of \cite{AzSaSp} wished to have this gap result. If $G$ is a compact group its convolution algebra $L^1(G)$ is an example of an \emph{amenable Banach algebra}; in contrast, it is shown in \cite{AzSaSp} that for many choices of compact group~$G$, the centre of $L^1(G)$ is \emph{not} amenable.\/\footnote{We shall not discuss amenability of Banach algebras in this article, but it has proved to be a fundamental and fruitful notion in certain areas of functional analysis; it may be viewed as a weakened version of ``homological dimension zero'', suitably interpreted.} One source of such examples is given by $G=\prod_{n=1}^\infty G_n$ where each $G_n$ is a finite non-abelian group, and this works by combining Theorem~\ref{t:ASSgap} with a tensor product argument to show that $\AMZL(\prod_{i=1}^n G_i) \geq (1+\delta)^n$ for each~$n$. Thus results for finite groups are used to provide compact groups that generate interesting examples of Banach algebras.
\end{rem}

Now let us return to Theorem~\ref{t:ASSgap}. Examining the original proof given in~\cite{AzSaSp}, one sees that it relies crucially on the following hard result of D.~A.~Rider.

\begin{thmlike}[(Rider; see {\cite[Lemma 5.2]{Rider_idem}})]{Theorem}
\label{t:rider-thm}
Let $K$ be a compact non-abelian group, equipped with a choice\footnotemark\ of Haar measure $\lambda$.
\footnotetext{Strictly speaking, Rider only states and proves this when $\lambda(K)=1$. However, the more general version stated here follows easily from the case $\lambda(K)=1$ by a rescaling argument.}
 Let $f$ be a finite linear combination of irreducible characters of $K$, regarded as an element of the convolution algebra $L^1(K,\lambda)$.
Let
$\norm{f}_1 \defeq \int_K \abs{ f(x) }\,d\lambda(x)$. 
If $f$ is an idempotent in $L^1(K,\lambda)$, then either $\norm{f}_1=1$ or $\norm{f}_1> 1 + 1/300$.
\end{thmlike}

Theorem~\ref{t:ASSgap}, with $\delta=1/300$, follows from Theorem~\ref{t:rider-thm} by taking $K=G\times G$, $\lambda$ to be counting measure and $f=\DZ(G)\in \ell^1(K)$.
Nevertheless, one would hope for a more direct proof of Theorem~\ref{t:ASSgap}, or a better value of $\delta$\/. Rider notes that the value $1/300$ in Theorem~\ref{t:rider-thm} could be improved, but his method does not give any indication of the likely order of magnitude of the optimal constant.\footnote{In some private calculations the present author has been able to get up to $1/80$, but not up to $1/10$.}
Furthermore, this existing proof of Theorem~\ref{t:ASSgap} ignores the fact that $\AMZL(G)$ is the norm of a very special element of $\ZCGG$.

In some recent work of the present author with M. Alaghmandan and E. Samei~\cite{ACS_AMZL-2cd}, the exact ZL-amenability constants were calculated for several families of finite groups. In that work, the smallest value obtained was $7/4$, achieved by the dihedral group of order~$8$, and no smaller value could be found. This demonstrated a large gap between the value of $\delta$ provided by Theorem~\ref{t:rider-thm}, and what one observed in practice.

In this paper we show that the inequality $\AMZL(G)\geq 1+1/300$ is a severe underestimate. Specifically, we prove the following new bound.

\begin{thm}\label{t:mainthm}
Let $G$ be a finite, non-abelian group. Then $\AMZL(G)\geq 7/4$.
\end{thm}

\begin{rem}
As previously noted, the ${\rm ZL}$-amenability constant of the dihedral group of order $8$ is exactly $7/4$. So the bound in Theorem~\ref{t:mainthm} is best possible.
Furthermore, we will see later in Example~\ref{eg:when AMZL=7/4} that there is an infinite sequence $(G_n)$ of finite $2$-groups, each one indecomposable as a direct product of smaller factors, such that $\AMZL(G_n)=7/4$ for all~$n$.
\end{rem}

We wish to immediately highlight two features of our approach to Theorem~\ref{t:mainthm}. Firstly, we completely bypass Rider's result; and secondly, we avoid any difficult structure theory for finite groups. In fact, all the group theory and character theory we need can be found in introductory texts. We need some results from~\cite{ACS_AMZL-2cd}, but the same comments apply there.

\subsection*{Structure and style of the paper}
In Section~\ref{s:prelim} we give a rapid definition of $\AMZL(G)$ and establish some basic properties that will be needed. These properties were already established in \cite{AzSaSp}, using a Banach-algebraic perspective; here we shall give complete proofs that only need Schur orthogonality for finite groups, so that the reader does not need to take results from \cite{AzSaSp} on trust.

Also, following an idea used by Azimifard, Samei and Spronk in their proof of Theorem~\ref{t:ASSgap}, we introduce an auxiliary invariant $\ASS(G)$ which is a lower bound for $\AMZL(G)$ and is easier to estimate from below.
In Section~\ref{s:trivial-centre} we obtain a new lower bound on $\ASS(G)$ for all groups with trivial centre (excluding the degenerate case of the $1$-element group). This is perhaps the main conceptually new idea in this paper that was missing from \cite{ACS_AMZL-2cd} and~\cite{AzSaSp}. Nevertheless, to get from here to the optimal constant in Theorem~\ref{t:mainthm} requires another ingredient, and that is to reduce the problem to the case of finite groups that are ``just non-abelian''. The precise statement is Theorem~\ref{t:bound for JNA cases}: the necessary definitions, and the proof of that theorem, take up all of Section~\ref{s:JNA cases}. We also need two technical results on just non-abelian groups (Lemma \ref{l:baby JN2} and Theorem~\ref{t:is-dihedral}): both can be extracted from the existing literature, but for the reader's convenience we provide full proofs in the appendix.

These sections, together with the Appendix, provide a complete proof of Theorem~\ref{t:mainthm}. Figure~\ref{fig:leitfaden} shows how the key results are assembled in this proof.
\begin{figure}[htpb]
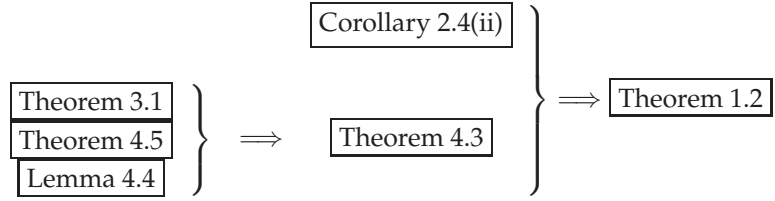

\[ \left. \begin{array}{lcc}
 & &  \fbox{Corollary \ref{c:AMZL-of-prod-and-quot}\ref{li:quot}}  \\
 & & \\
 \left.   \begin{tabular}{c}
 \fbox{Theorem~\ref{t:bound for trivial centre}} \\
 \fbox{Theorem~\ref{t:is-dihedral}} \\
  \fbox{Lemma~\ref{l:baby JN2}}
\end{tabular} \right\}  & \Longrightarrow
& \fbox{Theorem~\ref{t:bound for JNA cases}} 
\end{array} \right\} \Longrightarrow
\fbox{Theorem~\ref{t:mainthm}}
\]
\caption{Results needed to prove the main result}
\label{fig:leitfaden}
\end{figure}
For sake of completeness and interest: in Section \ref{s:ASS(perfect)} we use similar but easier ideas to those of Section~\ref{s:trivial-centre} to obtain a lower bound on $\ASS(G)$ for all perfect groups $G$; and in Section~\ref{s:closing} we list several natural questions that we have not yet resolved, concerning the set of possible values of $\AMZL(\cdot)$.

The style of this paper is intended to be somewhat expository. The aim, both here and in the earlier paper~\cite{ACS_AMZL-2cd}, has been to make the proofs accessible to researchers in abstract harmonic analysis who are interested in amenability constants of various Banach algebras, but who may not have much exposure to structure theory for finite groups. Therefore, the proofs are spelled out in some detail, and we have avoided more specialized results or techniques from finite group theory. On the other hand, this paper does not presume any knowledge of (or interest in) abstract harmonic analysis; so finite group theorists may find Theorem~\ref{t:mainthm} and its proof of some interest, and perhaps take the results obtained en route as a challenge to find sharper versions or better estimates.

\end{section}

\begin{section}{Preliminaries}\label{s:prelim}

\begin{subsection}{Notation and terminology}
We assume familiarity with basic character theory of finite groups over the complex field~$\Cplx$, up to the Schur orthogonality relations.
 All of this can be found in Isaacs's book \cite{Isaacs_CTbook}; alternatively, see \cite[Chapter 7]{Cohn_ed2vol2} or \cite{JamLie}.

 Given a finite group $G$, we denote by $\Conj(G)$ the set of conjugacy classes of $G$, and we write $\Irr(G)$ for the set of irreducible characters of~$G$ (working over $\Cplx$). The \dt{degree} of a character $\phi$ is just $\phi(e)$ (which equals the dimension of any representation affording that character) and we say that $\phi$ is \dt{linear} if $\phi(e)=1$.
If $C\in\Conj(G)$ and $\phi$ is a character on $G$, we write $\phi(C)$ for the value of $\phi$ on any element of $C$; thus $|\phi(C)|\leq\phi(e)$.

The element $\DZ(G)$ mentioned in the introduction is the \dt{diagonal element} or \dt{separability idempotent} for the commutative $\Cplx$-algebra $Z(\Cplx G)$. For the definition of these terms for general algebras over a field, see e.g.~\cite[Section 6.7]{Cohn_ed2vol3}; further explanation in the special case of the centre of a group algebra, and discussion of the connection with amenability for Banach algebras, can be found in \cite[Section~2.1]{ACS_AMZL-2cd}.

To keep our presentation self-contained, we will not use general properties of separability idempotents, and instead adopt a ``hands-on'' definition. First some terminology: if $\bbI$ is a non-empty finite indexing set, then by the ``natural'' $\ell^1$-norm on the vector space $\Cplx^\bbI$ we simply mean $\norm{a}_1 \equiv \norm{a}_{\ell^1(\bbI)} \defeq \sum_{i\in\bbI} |a_i|$.

\begin{dfn}\label{d:define diagonal}
Let $G$ be a finite group. If $\phi\in\Irr(G)$, we may regard $\phi\tp\phi$ as an element of $\Cplx (G\times G)$ (since the underlying vector space of $\Cplx (G\times G)$ is $\Cplx^{G\times G}$).
Let
\[
\DZ(G) \defeq  \sum_{\phi\in\Irr(G)} \frac{\phi(e)}{|G|}\phi \tp  \frac{\phi(e)}{|G|}\phi \in \Cplx(G\times G)\,.
\]
Then, equipping $\Cplx^{G\times G}$ with its natural $\ell^1$-norm, define the \dt{ZL-amenability constant of~$G$} to be the $\ell^1$-norm of $\DZ(G)$.
\end{dfn}

We could of course define $\AMZL(G)$ without explicitly naming $\DZ(G)$, but  some hereditary properties of $\AMZL$ with respect to taking products and quotients of groups are best understood in terms of corresponding properties of $\DZ$. This will be done in the next section.
\end{subsection}

\begin{subsection}{Hereditary properties of the ZL-amenability constant}
The next two lemmas are known results, if one uses the definition of $\DZ(G)$ that is given in \cite{AzSaSp} and~\cite{ACS_AMZL-2cd}. We shall give character-theoretic proofs that just use Definition~\ref{d:define diagonal}.

\begin{lem}\label{l:diagonal of group product}
Let $G$ and $H$ be finite groups. If we identify $\Cplx(G\times H\times G\times H)$ with $\Cplx(G\times G) \tp \Cplx(H\times H)$ in the natural way, then $\DZ(G\times H)$ is identified with $\DZ(G)\tp \DZ(H)$.
\end{lem}

\begin{proof}
This is a straightforward calculation using $\Irr(G\times H)=\{ \phi\tp\psi \st \phi\in\Irr(G), \psi\in\Irr(H)\}$ and $(\phi\tp\psi)(e_{G\times H}) = \phi(e_G)\psi(e_H)$.
\end{proof}

\begin{lem}\label{l:diagonal of quotient group}
Let $G$ and $H$ be finite groups, and suppose $q:G\to H$ is a surjective group homomorphism. Extend $q$ to an algebra homomorphism $Q:\Cplx G \to \Cplx H$ in the natural way. Then $(Q\tp Q)(\DZ(G))=\DZ(H)$.
\end{lem}

\begin{proof}
The map $\phi\mapsto \phi\circ q$ defines an injection $q^*: \Irr(H)\hookrightarrow \Irr(G)$.
Now
\[ Q\left(\sum_{y\in H} \sum_{x\in q^{-1}(y)} a_x\delta_x\right) = \sum_{y\in H} \left(\sum_{x\in q^{-1}(y)} a_x\right) \delta_y \qquad(a\in\Cplx G).\]
In particular, for each $\phi\in\Irr(H)$,
\[ Q\left( \frac{q^*(\phi)(e_G)}{|G|} q^*(\phi)\right) = \frac{\phi(e_H)}{|G|} |\ker(q)|  \phi = \frac{\phi(e_H)}{|H|} \phi\;. 
 \]

Let $\cR$ denote the set $\Irr(G)\setminus q^*(\Irr(H))$. Then
\[ Q\left( \sum_{\psi \in\Irr(G)\setminus \cR} \frac{\psi(e_G)}{|G|}\psi \tp   \frac{\psi(e_G)}{|G|}\psi \right) = \DZ(H)\,; \]
so to complete the proof, it suffices to show that $\cR\subseteq \ker(Q)$.

Our argument is indirect. Given $f\in\ZCG$, for each $y\in H$ we pick $x\in q^{-1}(y)$ and note that
\[ \delta_y * Q(f) = Q(\delta_x* f) = Q(f*\delta_x) = Q(f)* \delta_y\,.\]
Hence $Q(\ZCG)\subseteq \ZCH$. Since $\Irr(H)$ spans $\ZCH$, it now suffices to prove that
\[ \sum_{y\in H} Q(\theta)(y)\ \overline{\phi(y)} = 0 \qquad\text{for all $\theta\in \cR$ and all $\phi\in\Irr(H)$.} \]
Equivalently, it suffices to prove that for each $\theta\in \cR$ and each $\phi\in \Irr(H)$,
\[ 0 = \sum_{y\in H} \left( \sum_{x\in q^{-1}(y)} \theta(x) \right) \overline{\phi(y)} = \sum_{x\in G} \theta(x)\overline{(\phi\circ q)(x)}\,. \]
But this is immediate from Schur orthogonality for $\Irr(G)$.
\end{proof}

\begin{cor}[${\rm ZL}$-amenability constants of products and quotients]
\label{c:AMZL-of-prod-and-quot}\
\begin{YCnum}
\item\label{li:prod}
 Let $G_1$ and $G_2$ be finite groups. Then $\AMZL(G_1\times G_2)=\AMZL(G_1)\AMZL(G_2)$.
\item\label{li:quot}
 Let $G$ be a finite group and let $N \unlhd G$. Then $\AMZL(G/N)\leq \AMZL(G)$.
\end{YCnum}
\end{cor}

\begin{proof}
For any finite indexing sets $\bbI$ and $\bbJ$, the $\ell^1$-norm has the following property: for each $a\in \Cplx^\bbI$ and $b\in \Cplx^\bbJ$\/,
we have $\norm{a \tp b}_{\ell^1(\bbI\times\bbJ)} = \norm{a}_{\ell^1(\bbI)}\norm{b}_{\ell^1(\bbJ)}$\/.
Now, using Lemma \ref{l:diagonal of group product}, part~\ref{li:prod} follows.

To prove part~\ref{li:quot}, note that $Q:\Cplx G \to \Cplx H$ does not increase the $\ell^1$-norm of elements. The same is true for $Q\tp Q:\Cplx(G\times G) \to \Cplx(H\times H)$. Therefore $\norm {(Q\tp Q)(\DZ(G))}_1\leq \norm{\DZ(G)}_1$, and applying Lemma~\ref{l:diagonal of quotient group} does the rest.
\end{proof}

\begin{rem}\label{r:retread}
Both parts of Corollary~\ref{c:AMZL-of-prod-and-quot}
 were already proved in \cite[\S1]{AzSaSp}, using the abstract characterization of $\DZ(G)$ as the unique ``diagonal element'' or ``separability idempotent'' for the algebra $\ZCG$. Our approach has the benefit of just using Definition~\ref{d:define diagonal} and Schur orthogonality for finite groups, but the approach in \cite{AzSaSp} works in a somewhat broader context.
\end{rem}

\end{subsection}

\begin{subsection}{An explicit formula and some comments}
In \cite[Theorem 1.8]{AzSaSp} the authors give the following concrete formula for the ZL-amenability constant of a finite group:
\begin{equation}\label{eq:define AMZL}
\AMZL(G) = \frac{1}{|G|^2} \sum_{C\in\Conj(G)}\sum_{D\in\Conj(G)} |C| |D| \left\vert \sum_{\phi\in\Irr(G)} \phi(e)^2 \phi(C)\overline{\phi(D)}\right\vert\;.
\end{equation}
 We could have taken this as the \emph{definition} of $\AMZL(G)$, if we wished to avoid discussing $\DZ(G)$ and $\ell^1$-norms, but then the proof of Corollary~\ref{c:AMZL-of-prod-and-quot} would become even more opaque.

\begin{rem}
The presence of complex conjugates in Equation~\eqref{eq:define AMZL} may need some explanation. First, note that the same expression without any complex conjugates follows immediately from Definition~\ref{d:define diagonal} and the definition of the $\ell^1$-norm on $\Cplx(G\times G)$. Then observe that there is a bijection on $\Conj(G)$ which exchanges the conjugacy class of $x$ with that of $x^{-1}$; and if we denote this bijection by $D\longleftrightarrow D^{-1}$, then $|D^{-1}|\phi(D^{-1}) = |D|\overline{\phi(D)}$ for each $D\in\Conj(G)$ and each character $\phi$.
\end{rem}

\begin{eg}[ZL-amenability constants of finite abelian groups]
\label{eg:AMZL(abelian)}
Let $G$ be a finite abelian group. It was noted in \cite{AzSaSp} that in this case
$\DZ(G) = |G|^{-1} \sum_{x\in G} \delta_x \tp \delta_{x^{-1}}$, which clearly has $\ell^1$-norm exactly~$1$. (See the remarks in \cite[Section 2.1]{ACS_AMZL-2cd} for further explanation of this identity.) Here we give a more direct proof. 
By Equation~\eqref{eq:define AMZL}
\[ \AMZL(G)
 = \frac{1}{|G|^2}\sum_{x,y\in G} \left\vert
\sum_{\phi\in \widehat{G}} \phi(x)\overline{\phi(y)} \right\vert \,.
\]
By Schur (column) orthogonality, most terms in this sum vanish and summing the remaining ones gives
\[ \AMZL(G) = |G|^{-2} \sum_{x\in G} \sum_{\phi\in\widehat{G}} 1 = 1, \]
as required.
\end{eg}

In general it seems hard to get good \emph{lower} bounds on the expression in \eqref{eq:define AMZL}, since there could be considerable cancellation when we sum over~$\phi$, as in Example~\ref{eg:AMZL(abelian)}. A key idea in \cite{AzSaSp} is to introduce a minorant for $\AMZL(G)$ that is easier to estimate well from below. The authors of \cite{AzSaSp} do not give this invariant a name, but since it is the main object of study in Section~\ref{s:trivial-centre} it deserves some {\it ad hoc}\/ notation.

\begin{dfn}[An auxiliary minorant]\label{d:ASS}
For $G$ a finite group, we define
\begin{equation}
\ASS(G) = \frac{1}{|G|^2} \sum_{C\in\Conj(G)}\sum_{\phi\in\Irr(G)} |C|^2 |\phi(C)|^2 \phi(e)^2\ .
\end{equation}
\end{dfn}

\begin{prop}[Azimifard--Samei--Spronk, \cite{AzSaSp}]
\label{p:ASS(non-abelian)}
$\AMZL(G)\geq \ASS(G)\geq 1$ for every finite group~$G$. Moreover, if $G$ is non-abelian then $\ASS(G)>1$.
\end{prop}

\begin{proof}
This is demonstrated in the proof of Theorem 1.8 and Corollary 1.9 in \cite{AzSaSp}, but for sake of completeness we include the argument here.
Note that $\ASS(G)$ is obtained by considering the right-hand side of \eqref{eq:define AMZL} and only summing over the terms indexed by $\{(C,C) \st C\in\Conj(G)\} \subset \Conj(G)\times\Conj(G)$. Thus $\AMZL(G)\geq\ASS(G)$.
Moreover, since $|C|^2\geq |C|$, we have
\[ \begin{aligned}
\ASS(G) & \geq
 |G|^{-2} \sum_{C\in\Conj(G)}\sum_{\phi\in\Irr(G)} |C| |\phi(C)|^2 \phi(e)^2  \\
& = \frac{1}{|G|}\sum_{\phi\in\Irr(G)}  \phi(e)^2 \sum_{C\in\Conj(G)} \frac{|C|}{|G|} |\phi(C)|^2 & = 1\,,
\end{aligned} \]
where the final equality follows from the Schur orthogonality relations (row and column).
Finally, if $G$ is non-abelian, then there is at least one $C\in\Conj(G)$ and $\phi\in\Irr(G)$ such that $|C|^2|\phi(C)|^2 > |C| |\phi(C)|^2$, so that running the argument above we see that the inequality is strict.
\end{proof}

\begin{rem}
Combining Proposition~\ref{p:ASS(non-abelian)} with Rider's hard result (Theorem~\ref{t:rider-thm}) one sees that if $G$ is non-abelian then $\AMZL(G)\geq 1+ 1/300$, i.e.~we get a proof of Theorem~\ref{t:ASSgap}.
However, it is important to note that this does not tell us anything about
\[ \inf\{ \ASS(G) \st \text{$G$ is finite and non-abelian} \}. \]
Indeed, the present paper is unable to determine if this infimum is strictly greater than~$1$.
\end{rem}

A closer inspection of the proof of Proposition~\ref{p:ASS(non-abelian)} yields the inequality
\[ \ASS(G) - 1 \geq \frac{s(G)^2-s(G)}{|G|} \]
where $s(G) = \min \{ |C| \st C\in \Conj(G), |C| > 1 \}$.
However, in cases where $Z(G)=\{e\}$, we can do much better. This will be the topic of the next section.

\end{subsection}
\end{section}

\begin{section}{Bounds on $\ASS$ for groups with trivial centre}
\label{s:trivial-centre}
In this section we obtain a new lower bound for $\ASS(G)$ when $G$ has trivial centre. This is the key idea needed for a new proof of Theorem~\ref{t:ASSgap} with a much better constant than that of~\cite{AzSaSp}.

\begin{thm}[Lower bound on $\ASS(G)$ when $Z(G)=\{e\}$]
\label{t:bound for trivial centre}
Let $G$ be a finite group with trivial centre, and let $s$ be the smallest size of its non-trivial conjugacy classes. Then
\[ \ASS(G) -1 \geq \frac{s-1}{s+1} \cdot \frac{2s}{s+1} \geq \frac{4}{9}\ .\]
\end{thm}

\begin{rem}
The author has not found any groups with trivial centre for which $\ASS(G)=13/9$. Probably one can improve on this lower bound by investigating the examples with $s=2$ more closely.
\end{rem}

The proof of Theorem~\ref{t:bound for trivial centre} will occupy the rest of this section.
To streamline some of the formulas which follow, we
introduce the matrix $A:\Irr(G)\times \Conj(G) \to [0,1]$ defined by
\begin{equation}\label{eq:ignore phase}
A_{\phi,C} = |G|^{-1} |C| \ |\phi(C)|^2 \qquad(\phi\in\Irr(G),C\in\Conj(G)).
\end{equation}
Note that rows and columns of $A$ add up to~$1$.
We define
\[ \mu\defeq \max_{\phi\in\Irr(G)} A_{\phi,e} = \max_{\phi\in\Irr(G)} |G|^{-1} \phi(e)^2 \in (0,1).\]

\begin{prop}\label{p:BS-1st-step}
Let $G$ be a finite group with trivial centre, and let $s=s(G)$ be as defined in the statement of Theorem~\ref{t:bound for trivial centre}.
Then
\begin{equation}\label{eq:BS-1st-step}
\frac{\ASS(G)-1}{s-1} \geq \sum_{\phi\in\Irr(G)} A_{\phi,e}( 1- A_{\phi,e}) 
\geq \max(2\mu,1)(1-\mu).
\end{equation}
\end{prop}

\begin{proof}
We have
\[ \begin{aligned}
\ASS(G)
 &  =  \sum_{\phi\in\Irr(G)} \sum_{C\in \Conj(G)} |C| A_{\phi,C} A_{\phi,e} \\
 & \geq \sum_{\phi\in\Irr(G)} \left(A_{\phi,e}^2+ \sum_{C\in \Conj(G)\setminus\{e\}} sA_{\phi,C} A_{\phi,e} \right)
    & \text{(definition of $s$)} \\
 & = (1-s)\sum_{\phi\in\Irr(G)} A_{\phi,e}^2 + s\sum_{\phi\in\Irr(G)}\sum_{C\in \Conj(G)} A_{\phi,C} A_{\phi,e}  \\
 & = (1-s)\sum_{\phi\in\Irr(G)} A_{\phi,e}^2 + s\sum_{\phi\in\Irr(G)}  A_{\phi,e}     & \text{(rows of $A$ sum to $1$)} \\
 & = (s-1)\sum_{\phi\in\Irr(G)} A_{\phi,e}\left( 1- A_{\phi,e}\right) + \sum_{\phi\in\Irr(G)} A_{\phi,e} \,.
\end{aligned} \]
Since columns of $A$ sum to $1$, we obtain
the first inequality in \eqref{eq:BS-1st-step}.

To get the second inequality: fix $\psi\in\Irr(G)$ with $A_{\psi,e}=\mu = \max_{\phi\in\Irr(G)} A_{\phi,e}$. For each $\phi\in\Irr(G)\setminus\{\psi\}$ we have $1-A_{\phi,e}\geq 1-\mu$; also, since columns of $A$ sum to $1$,
\[   1-A_{\phi,e} = \sum_{\gamma\in\Irr(G)\setminus\{\phi\}} A_{\gamma,e} \geq A_{\psi,e}=\mu\,. \]
Hence
\[ \begin{aligned}
\sum_{\phi\in\Irr(G)} A_{\phi,e}\left( 1- A_{\phi,e}\right) 
& = \mu(1-\mu) + \sum_{\phi\in\Irr(G)\setminus\{\psi\}} A_{\phi,e}\left(1-A_{\phi,e} \right) \\
& \geq \mu(1-\mu) + \sum_{\phi\in\Irr(G)\setminus\{\psi\}} A_{\phi,e}\max(\mu,1-\mu) \\
& = \mu(1-\mu)+\max(\mu,1-\mu)(1-\mu) 
    & \text{(rows of $A$ sum to $1$)} \\ 
& =(1-\mu)\max(2\mu,1),
\end{aligned} \]
as required.
\end{proof}

\begin{eg}\label{eg:affine}
Let $q$ be a prime power $\geq 3$ and consider $G=\Aff(\bF_q)$, the affine group of the finite field with $q$ elements. The character theory of this group is well known, and gives
$\mu = |G|^{-1}(q-1)^2 = 1- q^{-1}$.
For this group $s=q-1$, so that
\begin{equation}\label{eq:crude bound for Aff}
\ASS(\Aff(\bF_q)) -1
  \geq (s-1)(1-\mu)\max(2\mu,1) 
  = (q-2) \cdot \frac{1}{q}\cdot 2\left(1-\frac{1}{q}\right);
\end{equation}
and the right-hand side of \eqref{eq:crude bound for Aff}, being a monotone increasing function of $q$, is minimized when $q=3$ (with minimum value $4/3$). In fact an explicit but tedious calculation shows that $\ASS(\Aff(\bF_q))=3-4q^{-1}$, so that our lower bound is a slight underestimate.
\end{eg}

\medskip
Example~\ref{eg:affine} shows that even within the class of groups with trivial centre, $1-\mu$ can be arbitrarily small. However, note that in this example, when $1-\mu$ is small $s$ is large. The following result shows that this happens more generally.

\begin{lem}\label{l:BS-2nd-step}
Let $G$ be a finite group. Then
$1\leq (1+s)(1-\mu)$.
\end{lem}
\begin{proof}
Choose $\psi$ with $A_{\psi,e}=\mu$ and choose $C$ with $|C|=s$.
Then $A_{\psi,C} \leq 1-\mu$ (since rows of $A$ sum to $1$), while
\[ A_{\phi,C} = \frac{s}{|G|} |\phi(C)|^2 \leq  \frac{s}{|G|} \phi(e)^2 = s A_{\phi,e} \quad\text{for all $\phi\in\Irr(G)$.} \] 
Since the columns of $A$ sum to $1$, this gives
\[ 1 = A_{\psi,C} + \sum_{\phi\neq\psi}A_{\phi,C} \leq (1-\mu) + s\sum_{\phi\neq \psi} A_{\phi,e} = (1-\mu)+s(1-\mu) \]
as required.
\end{proof}

\begin{rem}\

\begin{YCnum}
\item
The idea behind the proof of Lemma~\ref{l:BS-2nd-step} comes from considering cases such as Example~\ref{eg:affine}, where the ratio $\mu$ is large in the sense of being close to~$1$. Picking $\phi\in\Irr(G)$ with $A_{\phi,e}=\mu$ and $C\in\Conj(G)$ with $|C|=s$, one reasons as follows:
\begin{itemize}
\item since the column labelled by $e$ adds up to $1$,
 $A_{\phi,e}$ must be small for every $\phi\in \Irr(G)\setminus\{\psi\}$;
\item since the row labelled by $\psi$ adds up to~$1$, $A_{\psi,C}$ has to be small;
\item
in turn, this forces $A_{\phi,C}$ to be large for every $\phi\in\Irr(G)\setminus\{\psi\}$.
\end{itemize}

\item
There are groups for which equality is achieved in Lemma~\ref{l:BS-2nd-step}. For instance, if $G$ is the dihedral group of order~$6$, we have $s=2$, $\mu= 2^2/6 = 2/3$, $(s+1)(1-\mu) = 1$.

\item
Lemma~\ref{l:BS-2nd-step} implies that if a finite group has an irreducible character with ``relatively high degree'', not only must the group have trivial centre, but all conjugacy classes other than the identity have to be ``quite large''.
Such observations are probably not new, but we are unaware of any references in the literature which provide such inequalities.
\end{YCnum}
\end{rem}

\begin{proof}[Proof of Theorem~\ref{t:bound for trivial centre}]
Let $h:[0,1] \to [0,\infty)$ be the function $h(t) = (1-t)\max(2t,1)$. This is a continuous, decreasing function. By Proposition~\ref{p:BS-1st-step} and Lemma~\ref{l:BS-2nd-step},
\[ \ASS(A)-1 \geq (s-1)h(\mu) \geq (s-1)h\left(\frac{s}{s+1}\right) \ . \]
Now since $s/(s+1) \geq 2/3 > 1/2$\/,
\[ 
(s-1)h\left(\frac{s}{s+1}\right) = \frac{s-1}{s+1} \cdot \frac{2s}{s+1} \, ;
\]
and the right hand side is minimized when we take $s-2$, with value $4/9$.
\end{proof}

\begin{rem}
We could have tried the same approach for a non-abelian finite group $G$ with non-trivial centre.  The proof of Proposition~\ref{p:BS-1st-step} goes through much as before\footnotemark, but this time we merely obtain
\footnotetext{We need to know that $|\phi(x)|=\phi(e)$ for all $x\in Z(G)$ and all $\phi\in\Irr(G)$. This is an easy consequence of Schur's lemma, see e.g.~\cite[Lemma 2.25]{Isaacs_CTbook}.}
\[ \ASS(G)-1 \geq (s-1) \sum_{\phi\in\Irr(G)} \frac{\phi(e)^2}{|G|} \left( 1- \frac{|Z(G)| \phi(e)^2}{|G|}\right). \]
Now if we define $\mu=\max_{\phi\in\Irr(G)}|G|^{-1}\phi(e)^2$ as before, then $|Z(G)| \mu\leq 1$, but this time we can have equality, which means that some of the terms in the sum above will be zero. It is not clear how to obtain good lower bounds on the remaining contributions to the sum.
\end{rem}

\end{section}

\begin{section}{Bounds on $\AMZL$ for just non-abelian groups}
\label{s:JNA cases}

\begin{subsection}{JNA groups: definitions and two results}\label{ss:JNA background}

Throughout, we denote the \dt{derived subgroup} (or \dt{commutator subgroup}) of $G$ by $[G,G]$ rather than $G'$.

\begin{dfn}
We say that a quotient of a group $G$ is \dt{proper} if it is not equal to $G$ itself. A group $G$ is \dt{just non-abelian} (JNA for short) if it is non-abelian yet all its proper quotients are abelian.
\end{dfn}

\begin{rem}
Since every homomorphism from $G$ to an abelian group factors through the quotient homomorphism $G\to G/[G,G]$, a group $G$ is JNA if and only if all its non-trivial normal subgroups contain $[G,G]$. In alternative terminology: if $G$ is JNA then $[G,G]$ is the minimal non-trivial normal subgroup of~$G$.
\end{rem}

Recall that by Corollary~\ref{c:AMZL-of-prod-and-quot}\ref{li:quot}, if $H$ is a quotient of $G$ then $\AMZL(G)\geq \AMZL(H)$. Now if $G$ is a finite, non-abelian group, there exists a quotient of $G$ which is JNA. (The proof is a routine induction which we omit.) Hence
\[ \inf\{\AMZL(G)\st \text{$G$ is finite and non-abelian}\} =
  \inf\{\AMZL(G)\st \text{$G$ is finite and JNA}\}. \]
Therefore Theorem~\ref{t:mainthm} will follow from the following special case.

\begin{thm}\label{t:bound for JNA cases}
Let $G$ be a finite JNA group. Then $\AMZL(G)\geq 7/4$.
\end{thm}

Our proof of Theorem~\ref{t:bound for JNA cases} works by looking at various cases. 
The main point is that if $G$ is finite JNA and has non-trivial centre, then it can be handled by the main result of \cite{ACS_AMZL-2cd}. For this we use the following lemma, which is a simpler (but weaker) version of \cite[Lemma 12.3]{Isaacs_CTbook}.

\begin{lem}\label{l:baby JN2}
Let $G$ be a finite JNA group with non-trivial centre.
Then every non-linear $\phi\in\Irr(G)$ either has degree $1$ (i.e.~is linear) or has degree $|G:Z(G)|^{1/2}$.
Moreover, there is a prime $p$ such that $|G|$ is a power of $p$; $|[G,G]|=p$; and every conjugacy class has size $1$ or size~$p$.
\end{lem}

We will also need the following result, which the author was unable to find stated explicitly in the literature.

\begin{thm}[Folklore?]\label{t:is-dihedral}
Let $G$ be a finite JNA group which has trivial centre and has a conjugacy class of size $2$. Then $G$ is isomorphic to the dihedral group of order $2p$, for some odd prime~$p$.
\end{thm}

To keep this article as self-contained as possible, the Appendix includes
elementary proofs of Lemma~\ref{l:baby JN2} and Theorem~\ref{t:is-dihedral}.
In the rest of this subsection we offer some remarks on just non-abelian groups to provide extra context for the reader, and to acknowledge the sources for the ideas behind Lemma~\ref{l:baby JN2} and Theorem~\ref{t:is-dihedral}. The reader who just wants to get on with the proof of Theorem~\ref{t:bound for JNA cases} can skip to Section~\ref{ss:AMZL-of-JNA}.

\begin{rem}[JN2 groups and JM groups]\label{r:JM-JN2-history}
If $G$ is a JNA group with non-trivial centre, then $\{e\}\subsetneq [G,G]\subseteq Z(G)$, so that $G$ is $2$-step nilpotent.
Finite nilpotent JNA groups were classified by M. F. Newman \cite{Newman_JN2gp}, who christened them finite \dt{JN2 groups}; notable examples are the extra-special $p$-groups for any given prime~$p$.
His earlier paper \cite{Newman_JMgp} gives a classification of finite \emph{solvable} JNA groups with trivial centre, calling them finite \dt{JM groups}; examples include the affine groups of finite fields (cf.\ Example~\ref{eg:affine}). In general, every finite JM group is necessarily of the form
$\bF\rtimes {\bf D}$ for some finite field $\bF$ and some subgroup ${\bf D} \leq \bF^\times$. This is a necessary condition and not a sufficient one: see \cite[Theorem 6.4]{Newman_JMgp} for the full classification.
\end{rem}

\begin{rem}[Looking for a proof of Theorem~\ref{t:is-dihedral}]
If we knew the hypotheses of Theorem~\ref{t:is-dihedral} force $G$ to be solvable, we could extract this result from Newman's classification of JM groups. However, doing so is somewhat cumbersome: one has to analyze the conjugacy classes in ${\bf F}\rtimes {\bf D}$ to show that $|{\bf D}|=2$, explain why $|{\bf F}|$ is an odd prime, and so on. Alternatively, once we know $G$ is solvable, we could appeal to the results of Isaacs and Passman that were mentioned earlier; but this seems to require some form of Frobenius's theorem on malnormal subgroups (\cite[Theorem 7.2]{Isaacs_CTbook} or \cite[Chapter~7, Theorem 8.6]{Cohn_ed2vol2}). In both cases, of course, we would still need a separate argument to explain why the conditions in Theorem~\ref{t:is-dihedral} imply that the group is solvable! This is one reason why we chose to give a self-contained proof in the Appendix.
\end{rem}

\begin{rem}[Groups with two character degrees]
The present paper also owes an indirect debt to work of Isaacs and Passman
 on ``groups with two character degrees''. Rather than try to pick out those parts of the original papers which are relevant to the present article, the reader can find a good exposition in Chapter 12 of \cite{Isaacs_CTbook}. In particular, the dichotomy theorem contained in \cite[Lemma 12.3]{Isaacs_CTbook} implies that every finite JM group is Frobenius with abelian complement and kernel, so that its ZL-amenability constant  can be calculated explicitly using the methods of \cite[\S3.2]{ACS_AMZL-2cd}.
In~fact, the idea of looking at ZL-amenability constants of JNA groups arose from reading about this work in \cite[Chapter 12]{Isaacs_CTbook}, while seeking to extend the results of \cite{ACS_AMZL-2cd}.
\end{rem}

\end{subsection}

\begin{subsection}{The ZL-amenability constant in the JNA cases}
\label{ss:AMZL-of-JNA}
To fully exploit Lemma~\ref{l:baby JN2} and Theorem~\ref{t:is-dihedral}, we will apply the main theorem of \cite{ACS_AMZL-2cd}, stated below for clarity.

\begin{thm}[{\cite[Theorem 2.4]{ACS_AMZL-2cd}}]\label{t:ACSformula}
Let $m$ be an integer greater than or equal to $2$, and suppose
 $G$ is a finite non-abelian group such that $\phi(e)\in\{1,m\}$ for all $\phi\in\Irr(G)$. Then
\[ \AMZL(G)-1 = 2(m^2-1)\sum_{C\in\Conj(G)} \frac{|C|}{|G|} \left( 1- \frac{|C|}{|[G,G]|} \right). \]
\end{thm}

The following calculation was done in \cite[Example 3.13]{ACS_AMZL-2cd} for extraspecial $p$-groups, but works in greater generality. For convenience we give the proof.

\begin{prop}\label{p:AMZL of JN2}
Let $G$ be a finite JNA group with non-trivial centre. Then there is a prime $p$ and an integer $k\geq 1$ such that
\begin{equation}\label{eq:AMZL-of-JN}
 \AMZL(G) -1 = 2\left(1-\frac{1}{p^{2k}}\right)\left(1- \frac{1}{p}\right). 
\end{equation}
In particular, $\AMZL(G)\geq 7/4$.
\end{prop}

\begin{proof}
Let $p$ be the prime from Lemma~\ref{l:baby JN2} and let $k$ satisfy $|G:Z(G)|^{1/2}=p^k$.  By the lemma every $\phi\in\Irr(G)$ has degree $1$ or $p^k$, so $k$ is a positive integer (strictly positive since $G$ is non-abelian). Moreover, Lemma~\ref{l:baby JN2} also implies
 \[ \frac{|C|}{|[G,G]|}  = 
 \begin{cases} p^{-1} & \quad\text{if $C\in Z(G)$} \\ 1 & \quad\text{if $C\in\Conj(G)\setminus Z(G)$.}
\end{cases} \]
Applying Theorem~\ref{t:ACSformula} yields
\[ \begin{aligned}
\AMZL(G)-1
 & = 2(p^{2k}-1) \sum_{C\in\Conj(G)} \frac{|C|}{|G|} \left( 1- \frac{|C|}{|[G,G]|}\right) \\
 & = 2(p^{2k}-1) |Z(G)| \cdot \frac{1}{|G|} \left( 1- \frac{1}{p}\right) \\
 & = 2\frac{(p^{2k}-1)}{|G:Z(G)|} \left(1-\frac{1}{p}\right) \\
 & = 2\left( 1- \frac{1}{p^{2k}}\right)\left(1-\frac{1}{p}\right).
\end{aligned} \]
To finish: note that the right-hand side of Equation~\eqref{eq:AMZL-of-JN} is an increasing function of $k$ and of~$p$. Taking $k=1$ and $p=2$ gives the minimum value $3/4$, as required.
\end{proof}

\begin{eg}[JNA groups achieving the minimal ZL-amenability constant]\label{eg:when AMZL=7/4}
The calculations in  \cite[Example 3.13]{ACS_AMZL-2cd} show that the only extraspecial $p$-groups with ZL-amenability constant equal to $7/4$ are the dihedral and quaternion groups of order~$8$.
However, for any $n\in\Nat$ the group denoted in \cite{Newman_JN2gp} by $M(2^n)$ is a JNA group of order $2^{n+2}$ whose centre has order $2^n$. The proof of  Proposition~\ref{p:AMZL of JN2} then yields $\AMZL(M(2^n)) = 7/4$.
\end{eg}

We now turn to JNA groups with trivial centre. If $G$ is such a group and it has a conjugacy class of size~$2$, then (Theorem~\ref{t:is-dihedral}) $G$ is isomorphic to the dihedral group $C_p\rtimes C_2$ for some odd prime~$p$.
As shown in \cite[Section 3.1]{ACS_AMZL-2cd}, Theorem~\ref{t:ACSformula} applies to $C_p\rtimes C_2$ and straightforward calculations yield
\begin{equation}\label{eq:dihedral-cases}
\AMZL(C_p\rtimes C_2)  = 1+3\left(1-\frac{1}{p}\right)^2 \geq 1+ 3\left(1-\frac{1}{3}\right)^2 = \frac{7}{3}.
\end{equation}

It remains to deal with JNA groups where every conjugacy class, except that of $e$, has size $\geq 3$. (Recall that this class includes all finite simple non-abelian groups!) Here, it is no longer clear that we can exploit the results of \cite{ACS_AMZL-2cd}, but instead we can turn to the results of Section~\ref{s:trivial-centre}.

\begin{prop}\label{p:generic}
Let $G$ be a finite group with at least two elements and trivial centre. Suppose that $G$ has no conjugacy classes of size $2$. Then $\AMZL(G)\geq \ASS(G) \geq 7/4$.
\end{prop}

\begin{proof}
Let $s= \min \{ |C| \st C\in \Conj(G)\setminus\{e\} \}$. The conditions put on our group ensure that $s\geq 3$. Then, by Theorem~\ref{t:bound for trivial centre},
\[ \begin{aligned}
\ASS(G)-1
 & \geq \frac{s-1}{s+1} \frac{2s}{s+1} \\
 & = 2\left(1- \frac{2}{s+1}\right)  \left(1-\frac{1}{s+1}\right) \\
 & \geq  2\left(1- \frac{2}{4}\right)  \left(1-\frac{1}{4}\right) & = \frac{3}{4}\,,
\end{aligned} \]
as claimed.
\end{proof}

\begin{proof}[Proof of Theorem~\ref{t:bound for JNA cases}]
Combine Proposition~\ref{p:AMZL of JN2}, Proposition~\ref{p:generic} and the inequality~\eqref{eq:dihedral-cases}.
\end{proof}

\begin{rem}
When $G$ has trivial centre and has a conjugacy class of size $2$, the direct application of Theorem~\ref{t:bound for trivial centre} merely gives $\AMZL(G)\geq \ASS(G)\geq 13/9$, which is not good enough. This is why we had to deal with such cases separately (and also use the JNA condition, which is not used in Proposition~\ref{p:generic}).
\end{rem}

\end{subsection}

\end{section}

\begin{section}{A lower bound for perfect groups}\label{s:ASS(perfect)}
Since every non-abelian simple group is JNA, and since we proved Theorem~\ref{t:mainthm} by passing from the given non-abelian finite group to some JNA quotient, it is natural to wonder what can be said about the ZL-amenability constant of a finite simple group. In this section we obtain a lower bound that is far better than what we obtain from Theorem~\ref{t:bound for trivial centre}, although it is probably far short of the true infimum.
 In fact, our bound holds for all finite perfect groups, and our approach does not require any structure theory.

\begin{prop}[An alternative bound for $\ASS$]\label{p:dual bound for ASS}
Let $G$ be a finite group, and let $m=\min\{ \phi(e) \st \phi\in\Irr(G)\text{ and } \phi(e)>1\}$. Then
\[ \ASS(G)-1 \geq (m^2-1) \sum_{C\in\Conj(G)} \frac{|C|}{|G|}\left( 1 - \frac{|C|}{|[G,G]|} \right). \]
Equality holds if every non-linear irreducible character of $G$ has degree~$m$.
\end{prop}

\begin{proof}
This is like the proof of \cite[Theorem 2.4]{ACS_AMZL-2cd}, only simpler; in some sense it is dual to the argument used in the proof of Proposition~\ref{p:BS-1st-step}.

As in Equation~\ref{eq:ignore phase}, let $A_{\phi,C}:= |G|^{-1}|C| |\phi(C)|^2$. Let $\veps$ denote the trivial character of $G$, sometimes called the \dt{augmentation character}, which takes the value $1$ on every conjugacy class. Then
\[
\ASS(G) =\sum_{C\in\Conj(G)} A_{\veps,C}\sum_{\phi\in \Irr(G)} \phi(e)^2A_{\phi,C}  \,.
\]

Let $\cL=\{ \phi \in\Irr(G)\st \phi(e)=1\}$.
Then for each $C\in\Conj(G)$,
\[ \begin{aligned}
\sum_{\phi\in\Irr(G)} \phi(e)^2 A_{\phi,C} 
& \geq \sum_{\phi\in\cL} A_{\phi,C}  + m^2 \sum_{\phi\in\Irr(G)\setminus\cL} A_{\phi,C} 
  & \text{(definition of $m$)} \\
 & = (1-m^2) \sum_{\phi\in\cL} A_{\phi,C} +  m^2 \sum_{\phi\in\Irr(G)} A_{\phi,C} \\
 & = (1-m^2)\sum_{\phi\in\cL} A_{\phi,C}  + m^2  &\text{(columns of $A$ sum to $1$)} \\
 & = (1-m^2) |\cL| A_{\veps,C}  + m^2  &\text{(definition of $\cL$).} \\
\end{aligned} \]

Multiplying this inequality by $A_{\veps,C}$ and summing over all $C$, 
\[ \ASS(G) \geq \sum_{C\in\Conj(G)} A_{\veps,C} \;+\; (m^2-1) \sum_{C\in\Conj(G)} A_{\veps,C}( 1- |\cL|A_{\veps,C}).\]
Since rows of $A$ sum to $1$,
 and $|\cL|$ equals the order of the quotient group $G/ [G,G]$, the result follows.
\end{proof}

The following result is surely folklore, but we include a proof since it is straightforward.

\begin{lem}\label{l:big conjugacy class}
Let $G$ be a finite group. Suppose it has a conjugacy class $C$ of size $|G|/2$. Then $K=G\setminus C$ is an index $2$ subgroup of $G$.
In particular, $G$ is not perfect.
\end{lem}

\begin{proof}
We divide the proof into several steps.

\para{Step 1} 
Let $t\in C$. Then its centralizer $Z_G(t)$ has order $|G|/|C|=2$. Since $\gen{t}\subseteq Z_G(T)$, $t$ is an involution and $Z_G(t)=\{1,t\}$.

\para{Step 2} If $x\in G$ is an involution, and $y\in C$ and $xy$ is an involution, then $x=e$ or $x=y$. 

\noindent[{\it Proof}\/:
we have $xy=(xy)^{-1}=y^{-1}x^{-1}=yx$, so $x\in Z_G(y)$; now use Step~1.]

\para{Step 3}
Let $x,y\in C$. Then $xy\in K$.

\noindent[{\it Proof}\/:
if $x=y$ then $xy=e$; while if $x$ and $y$ are distinct elements of $C$, then by Step~2, $xy$ cannot be an involution, so in particular $xy\notin C$.]

\para{Step 4}
Let $t\in C$. Then the functions $L_t, R_t:G\to G$ defined by $L_t(x)=tx$ and $R_t(x)=xt$ are involutive bijections $C\leftrightarrow K$.

\noindent[{\it Proof}\/:
by Step 3, $L_t(C)\subseteq K$; and since $t$ is an involution, $L_t$ is injective. So $|L_t(C)|\leq K$. But $|K| = |G\setminus C| = |C|$, hence $L_t:C\to K$ is surjective. The argument for $R_t$ is similar.]

\para{Step 5} $K$ is closed under multiplication.

\noindent[{\it Proof}\/:
let $a,b\in K$. Pick $t\in C$; by Step 4, $at,tb\in C$. Since $t$ is an involution, $ab=(at)(tb)\in C\cdot C \subseteq K$ by Step~3.]

\para{Final step} Since $C$ is the conjugacy class of an involution, it is closed under inversion. Hence $K=G\setminus C$ is also closed under inversion. Combined with Step 5, this shows $K$ is a subgroup of $G$; and since $K$ is the union of conjugacy classes, it is normal. Finally, since $G/K$ is abelian of order $2$, $G$ is not perfect.
\end{proof}

\begin{rem}
In Step 1 of the proof we showed, in effect, that $C$ consists of self-centralizing involutions. So if $t\in C$ then $\gen{t}$ is \emph{malnormal} in $G$, hence by Frobenius's theorem (\cite[Theorem 7.2]{Isaacs_CTbook} or \cite[Chapter~7, Theorem 8.6]{Cohn_ed2vol2}) it admits a complement, and Lemma~\ref{l:big conjugacy class} follows.
However, Frobenius's theorem remains quite a deep result and seems excessive in this instance.
\end{rem}

\begin{thm}\label{t:ASS(perfect)}
Let $G$ be a finite perfect group. Then $\ASS(G)\geq 3$.
\end{thm}

\begin{proof}
If $C\in\Conj(G)$ then $|C|\leq |G|/3$ (since $|C|$ is a proper factor of $|G|$, and it cannot be $|G|/2$ by Lemma~\ref{l:big conjugacy class}).
Applying Proposition~\ref{p:dual bound for ASS} and using $[G,G]=G$ yields
\[ \ASS(G) -1 \geq (m^2-1) \sum_{C\in\Conj(G)} \frac{|C|}{|G|} \left(1 - \frac{1}{3}\right) = \frac{2(m^2-1)}{3}\,.\]
Since $m\geq 2$ the result follows.
\end{proof}

\begin{rem}\label{r:simple-no-2dim-irrep}
If $G$ is simple one can improve this lower bound. A finite simple group has no irreducible $2$-dimensional representations over $\Cplx$ (see e.g.~\cite[Corollary 22.13]{JamLie});
hence $m\geq 3$ when $m$ is as defined in Proposition~\ref{p:dual bound for ASS}. Running the argument in the proof of Theorem~\ref{t:ASS(perfect)} then yields $\ASS(G)\geq 19/3$.

Moreover: since quotients of perfect groups are perfect, every finite perfect group $G$ has a quotient group $K$ which is non-abelian and simple. It follows that if $G$ is a finite perfect group,
\[ \AMZL(G)\geq \AMZL(K) \geq \ASS(K)\geq \frac{19}{3}\,.\]
While this argument does not imply $\ASS(G)\geq 19/3$ for all finite perfect groups, we know of no finite perfect group whose associated minorant is less than $19/3$. 
\end{rem}

\end{section}

\begin{section}{Questions for future investigation}
\label{s:closing}

Our first two questions concern the associated minorant $\ASS(G)$. Although it is technically easier to deal with than $\AMZL(G)$ for particular groups, we do not know if it has the same good hereditary properties.

\begin{qn}
Is there a finite group $G$ and $N\lhd G$ such that $\ASS(G) < \ASS(G/N)$?
\end{qn}

If the answer is negative, then we can obtain lower bounds on $\ASS(G)$ by passing to JNA quotients, and hence the next question would have a \emph{positive} answer.

\begin{qn}
Is $\inf\{ \ASS(G) \colon \text{$G$ non-abelian and finite} \}$ strictly greater than $1$?
\end{qn}

To motivate the next two questions, we present a small consequence of our earlier work on the JN2 cases.

\begin{thm}\label{t:nilpotent}
Let $G$ be a finite nilpotent group. If $\AMZL(G) < 59/27$ then $G$ is the product of a $2$-group with an abelian group.
\end{thm} 

\begin{proof}
Since $G$ is nilpotent it factorizes as a direct product of finite $p$-groups, say $G=H_1\times\dots\times H_m$.
(More precisely, it is the direct product of its Sylow $p$-subgroups; see~e.g.\
\cite[Section 9.8, Theorem 5]{Cohn_ed2vol1}.)

Suppose at least one of these groups, say $H_1$, is a non-abelian $p$-group where $p\geq 3$. Then $H_k$ admits a JNA quotient $K$, and $K$ must also be a $p$-group; the proof of Proposition~\ref{p:AMZL of JN2} then implies that
\[ \AMZL(K)-1 \geq 2( 1- p^{-2})(1-p^{-1}) \geq 2\left(1-\frac{1}{9}\right)\left(1-\frac{1}{3}\right) = \frac{32}{27}\,. \]
But since $\AMZL(G)\geq\AMZL(H_1)\geq\AMZL(K)$, this contradicts the assumption on $\AMZL(G)$.
Thus each of the groups $H_1,\dots, H_m$ is either abelian or a $2$-group, and the result follows.
\end{proof}

\begin{qn}\label{q:small AMZL forces nilpotent}
Does Theorem~\ref{t:nilpotent} remain true if one drops the word ``nilpotent''?
\end{qn}

Note that a positive answer to Question~\ref{q:small AMZL forces nilpotent} would yield a positive answer to the next question.

\begin{qn}
Is $\inf\{\AMZL(G) \colon \text{$G$ a finite, non-nilpotent group}\}$ strictly greater than $7/4$? Can we at least show that a group with ZL-amenability constant $7/4$ has to be nilpotent?
\end{qn}

\begin{rem}
An explicit calculation shows that $\ASS(C_3\rtimes C_2)=5/3$. Therefore, with our current methods we have no hope of proving results of the form ``if $G$ has trivial centre then $\AMZL(G) \geq 2$'', because without the JNA condition the best we can do for groups with trivial centre is to use the lower bound provided by $\ASS(G)$.
\end{rem}

Moving away from questions of lower bounds, we could look at the structure of the set
\[ \cZL=\{ \AMZL(G) \colon \text{$G$ a finite group}\}. \]
  of possible values for ZL-amenability constants. 

\begin{qn}
Is $\cZL$ closed? If not, does it contain any non-isolated points?
\end{qn}

The results summarized in \cite[Section 3.4]{ACS_AMZL-2cd} show that there exist strictly increasing sequences in $\cZL$ which converge. The same results suggest that these limiting values are limit points of $\cZL$, and hence that $\cZL$ is not closed, but they are not conclusive.

\begin{qn}
Does $\cZL$ contain any strictly decreasing sequences? That is, does there exist a sequence $(G_n)$ of finite groups such that $\AMZL(G_{n+1})< \AMZL(G_n)$ for all $n$?
\end{qn}

Our final question connects the study of $\AMZL(G)$ back to the world of infinite-dimensional Banach algebras.
\begin{qn}
 Consider an inverse system of finite groups
\[ \dots G_3 \to G_2 \to G_1 \to G_0 \]
in which all connecting maps are surjective group homomorphisms.
 If $\sup_n\AMZL(G_n) < \infty$, is the inverse limit $\varprojlim_n G_n$ always virtually abelian?
\end{qn}

If the answer is positive, then as noted in \cite{AzSaSp}, we could prove the following: if $G$ is a profinite group which is not virtually abelian, then the centre of $L^1(G)$ is non-amenable.
\end{section}

\subsection*{Acknowledgements}
The main ideas needed for the proof of Theorem~\ref{t:mainthm} were originally worked out in 2013, while the author worked at the University of Saskatchewan.
 He thanks M. Alaghmandan and E. Samei for their interest and encouragement.

Other important parts, such as an early version of Theorem~\ref{t:is-dihedral}, were obtained while attending the {\it Conference in Abstract Harmonic Analysis}\/ held in Granada, May 2013. The author thanks the conference organizers for an engaging conference and a stimulating atmosphere.  During this time the author was supported by NSERC Discovery Grant 402153-2011.

The results of Section~\ref{s:ASS(perfect)} were obtained during the author's present position at Lancaster University, which is also where this paper took shape as various proofs were streamlined, rearranged and improved.
Thanks are due to the referee for an attentive reading and constructive feedback, in particular for drawing his attention to the reference cited in Remark~\ref{r:simple-no-2dim-irrep}.

\appendix

\begin{section}{The proofs of Lemma~\ref{l:baby JN2} and Theorem~\ref{t:is-dihedral}}
These proofs are included for the convenience of readers who are more familiar with harmonic analysis and Banach algebras than finite group theory.
To clarify the proof of Lemma~\ref{l:baby JN2}, it is helpful to divide the statement of the lemma into smaller pieces.

\para{Restatement of Lemma~\ref{l:baby JN2}}
{\it
Let $G$ be a finite JNA group which has non-trivial centre.
 Then there is a prime $p$ such that:
\begin{YCnum}
\item\label{li:order of G}
 $|G|$ is a power of $p$;
\item\label{li:order of [G,G]}
 $|[G,G]|=p$;
\item\label{li:conj class sizes}
 every conjugacy class has size $1$ or size $p$;
\item\label{li:2cd}
 every non-linear $\phi\in\Irr(G)$ has degree $|G:Z(G)|^{1/2}$.
\end{YCnum}
}

\para{\it Proof of \ref{li:order of G}}
Since $Z(G)$ is non-trivial, $G/Z(G)$ is abelian. So $G$ is $2$-step nilpotent. Hence there exist subgroups $H_1,\dots, H_s$ which each have prime power order, such that $G\iso H_1\times \dots \times H_s$. (More precisely, $G$ factorizes as the direct product of Sylow subgroups, see~e.g.\
\cite[Section 9.8, Theorem 5]{Cohn_ed2vol1}.)
At least one of these subgroups, say $H_1$\/, must be non-abelian; but then since $G$ quotients onto $H_1$, the JNA condition forces $G=H_1$.

\para{\it Proof of \ref{li:order of [G,G]}}
By \ref{li:order of G}, $G$ has order $p^N$ for some prime $p$ and some $N\geq 2$. So $|Z(G)|=p^m$ for some $1\leq m\leq N-1$. By the JNA condition $[G,G]$ is contained in every non-trivial subgroup of $Z(G)$. Hence $[G,G]$ is an abelian $p$-group with no non-trivial proper subgroups, so it must be cyclic of order~$p$.

\para{\it Proof of \ref{li:conj class sizes}}
Let $C\in\Conj(G)\setminus Z(G)$. Since $|C|$ divides $|G|$ it is a power of~$p$. On the other hand, there is an injection $C\to [G,G]$ (fix $x\in C$ and consider the map $C\to [G,G]$, $y\mapsto yx^{-1}$). By part~\ref{li:order of [G,G]}, $|[G,G]|=p$, and so $|C|=1$ or $|C|=p$.

\medskip

\para{\it Proof of \ref{li:2cd}}
Let $\cL$ denote the set of linear characters on $G$, and let $C\in\Conj(G)\setminus Z(G)$. By Schur orthogonality (column sums)
\[ \frac{|G|}{|C|} -  \sum_{\phi\in\Irr(G)\setminus\cL} |\phi(C)|^2 = \sum_{\psi\in \cL} |\psi(C)|^2 = |\cL| \,.\]
On the other hand: since any finite abelian group has the same cardinality as its dual group,
\[ |\cL| = |(G/ [G,G])^{\wedge} | =   |G/ [G,G] | =  \frac{|G|}{|[G,G]|} . \]
By \ref{li:order of [G,G]} and \ref{li:conj class sizes}, $|C|=p=|[G,G]|$. Combining this with the previous two identities forces
\[ \sum_{\phi\in\Irr(G)\setminus\cL} |\phi(C)|^2=0, \]
which in turn forces
 $\phi(C)=0$ for all $\phi\in \Irr(G)\setminus\cL$.

Now let $\phi\in  \Irr(G)\setminus\cL$. By the previous paragraph, $\phi(C)=0$ for all $C\in\Conj(G)\setminus Z(G)$. Hence, by Schur orthogonality (row sums),
\[ |G| = \sum_{x\in Z(G)} |\phi(x)|^2 = |Z(G)| \phi(e)^2 \,.\]
Rearranging yields the desired identity, and the proof of the lemma is complete.\hfill$\Box$

\medskip
Now we turn to Theorem~\ref{t:is-dihedral}, which we paraphrase as follows.

\para{Theorem~\ref{t:is-dihedral}, paraphrased}
{\it Let $G$ be a finite JNA group which has trivial centre and has a conjugacy class of size $2$. Then $G$ contains an involution $t$ and an element $r$ of odd prime order $p$, such that $|G|=2p$ and $rt=tr^{-1}$.}

\medskip
Originally the present author re-invented a proof of this result, but the argument which follows is streamlined slightly using some ideas from
Newman's paper~\cite{Newman_JMgp}.
To illustrate where we use the various hypotheses on $G$, the proof is broken up into several lemmas.

\begin{lem}\label{l:keep-or-flip}
Let $G$ be a finite group with a conjugacy class $\{a,\abar\}$ of size $2$. Then for each $g\in G$, conjugation by $g$ either fixes both $a$ and $\abar$ or it swaps them. In particular, $a$ and $\abar$ commute.
\end{lem}

\begin{proof}
When a group acts on a set of size two, each element either acts trivially or swaps the two points. This gives us the first claim in the lemma.
The second claim follows by considering the conjugation action of $a$ on $\{a,\abar\}$.
\end{proof}

\begin{lem}\label{l:[G,G] is cyclic}
Let $G$ be a finite JNA group with a conjugacy class $\{a,\abar\}$ of size $2$.
Put $r=\abar a^{-1}\neq\{e\}$. Then $[G,G]=\gen{r}$, and the order of $r$ is prime. Moreover, $r$ has order $2$ if and only if it is central.
\end{lem}

\begin{proof}
Let $g\in G$. By Lemma~\ref{l:keep-or-flip}, either $grg^{-1}=r$ or $grg^{-1} = a (\abar)^{-1}=(\abar)^{-1}a^{-1}=r^{-1}$. Thus the conjugation map $x\mapsto gxg^{-1}$ sends $r$ to $r^{\pm 1}$, and since $r^{-1}\in \gen{r}$ this shows $\gen{r}\lhd G$.
So by the JNA condition $[G,G]\subseteq\gen{r}$. On the other hand, let $g\in G$ be such that $gag^{-1} = \abar$; then $r= gag^{-1}a^{-1} \in [G,G]$. Thus $[G,G]=\gen{r}$.

Let $p$ be a prime which divides the order of $r$. Then $\gen{r^p} <\gen{r}=[G,G]$. Moreover $\gen{r^p}\lhd G$, since the only conjugates of $r^p$ are $r^p$ itself and $r^{-p}$, which both belong to $\gen{r^p}$. So by the JNA condition, $r^p=e$. Thus $r$ has order~$p$.

Finally: since the conjugacy class of $r$ is $\{r^{\pm 1}\}$, $r$~is central if and only if $r=r^{-1}$, i.e.~if and only if $r^2=e$.
\end{proof}

Choose some $t\in G$ such that $\abar=tat^{-1}$. Let $H=Z_G(a)=Z_G(\abar)$: this has index $2$ in $G$, so is normal, and $G=H\sqcup Ht$. Note that $\gen{r} \subseteq H$, by another application of Lemma~\ref{l:keep-or-flip}.

\begin{lem}\label{l:self-centralizing}
Let $G$ be as above, and suppose in addition that $Z(G)=\{e\}$. Then $H=[G,G]$.
\end{lem}

\begin{proof}
By our previous remarks,  $[G,G] = \gen{r} \subseteq H$.
So it suffices to prove $|H|\leq |[G,G]|$, and we can do this by exhibiting an injection $H\to [G,G]$. (Note that this injection will not be a left inverse to the inclusion of sets $[G,G]\hookrightarrow H$.)

Consider the function $\theta: H\to [G,G]$ defined by $\theta(h)=tht^{-1}h^{-1}$. Since $[G,G] \subseteq Z_G(H)$,
\[
t(hk)t^{-1} = (tht^{-1})(tkt^{-1}) = \theta(h)h\theta(k)k=\theta(h)\theta(k)hk \quad\text{for all $h,k\in H$.}
\]
Hence $\theta(hk)=\theta(h)\theta(k)$ for all $h,k\in H$, i.e.~$\theta$ is a homomorphism.
Now consider the subgroup
\[ K=\ker(\theta)= \{h\in H \colon ht=th\} = H \cap Z_G(t).\]
Since $K$ is the kernel of a homomorphism $H\to G$, $K\unlhd H$, and so $K\lhd G$ since $G=H\cup Ht$ and $K\subseteq Z_G(t)$.
Suppose $K\neq\{e\}$. Then $K\supseteq [G,G]$, by the JNA condition, and so $r\in K$ by Lemma~\ref{l:[G,G] is cyclic}. But this would imply $rt=tr$, and since $r\in Z_G(H)$ it would follow from $G=H\cup Ht$ that $r\in Z(G)$, contradicting the assumption that $Z(G)=\{e\}$. Thus $K=\{e\}$, so $\theta$ is injective, as required.
\end{proof}

\begin{proof}[Finishing the proof of Theorem~\ref{t:is-dihedral}]
Let $G$ be finite JNA with trivial centre and a conjugacy class of size~$2$. Combining the previous lemmas, we obtain $r,t\in G$ such that:
$trt^{-1}=r^{-1}$;
 $r$ has odd prime order;
and $G=\gen{r,t}$ (this follows from Lemma~\ref{l:self-centralizing} and the fact $H=\gen{r}$).
It only remains to show that $t^2=e$.
Since $trt^{-1}=r^{-1}$, $t^2$ commutes with $r$, and hence with everything in $G=\gen{r,t}$. So $t^2\in Z(G)=\{e\}$, and the proof of Theorem~\ref{t:is-dihedral} is complete.
\end{proof}

\end{section}


\vfill
\noindent Yemon Choi

\noindent
Department of Mathematics and Statistics\\
Fylde College, Lancaster University\\
Lancaster, United Kingdom LA1 4YF

\medskip
\noindent
Email: \texttt{y.choi1@lancaster.ac.uk}

\end{document}